\documentclass[reqno, 12pt,a4letter]{amsart}
\usepackage{amsmath,amsxtra,amssymb,latexsym, amscd,amsthm}
\usepackage[utf8]{inputenc}
\usepackage[mathscr]{euscript}
\usepackage{mathrsfs}
\usepackage[english]{babel}
\newtheorem {theorem}{Theorem}[section]

\newtheorem {lemma}{Lemma}[section]

\newtheorem*{Claim}{Claim}
\newtheorem {example}{Example}[section]
\newtheorem {definition}{Definition}[section]
\newtheorem {remark}{Remark}[section]

\def\ar{a\kern-.370em\raise.16ex\hbox{\char95\kern-0.53ex\char'47}\kern.05em}
\def\ees{{\accent"5E e}\kern-.385em\raise.2ex\hbox{\char'23}\kern-.08em}
\def\eex{{\accent"5E e}\kern-.470em\raise.3ex\hbox{\char'176}}
\def\AR{A\kern-.46em\raise.80ex\hbox{\char95\kern-0.53ex\char'47}\kern.13em}
\def\EES{{\accent"5E E}\kern-.5em\raise.8ex\hbox{\char'23 }}
\def\EEX{{\accent"5E E}\kern-.60em\raise.9ex\hbox{\char'176}\kern.1em}
\def\ow{o\kern-.42em\raise.82ex\hbox{
  \vrule width .12em height .0ex depth .075ex \kern-0.16em \char'56}\kern-.07em}
\def\OW{O\kern-.460em\raise1.36ex\hbox{
\vrule width .13em height .0ex depth .075ex \kern-0.16em \char'56}\kern-.07em}
\def\UW{U\kern-.42em\raise1.36ex\hbox{
\vrule width .13em height .0ex depth .075ex \kern-0.16em \char'56}\kern-.07em}
\def\DD{D\kern-.7em\raise0.4ex\hbox{\char '55}\kern.33em}

\title{Optimality conditions for minimizers at infinity in polynomial programming}

\author{TI\EES N-S\OW N PH\d{A}M}
\address{Department of Mathematics, University of Dalat, 1 Phu Dong Thien Vuong, Dalat, Vietnam}
\email{sonpt@dlu.edu.vn}

\thanks{The author was partially supported by Vietnam National Foundation for Science and Technology Development (NAFOSTED), grant 101.04-2016.05}

\date{ \today}

\begin{document}
\maketitle

\begin{abstract}
In this paper we study necessary optimality conditions for the optimization problem
$$\textrm{infimum}f_0(x) \quad \textrm{ subject to } \quad x \in S,$$
where $f_0 \colon \mathbb{R}^n \rightarrow  \mathbb{R}$ is a polynomial function and $S \subset \mathbb{R}^n$ is a set defined by polynomial inequalities.
Assume that the problem is bounded below and has the Mangasarian--Fromovitz property at infinity. 
We first show that if the problem does {\em not} have an optimal solution, then a version at infinity of the Fritz-John optimality conditions holds. From this we derive a version at infinity of the Karush--Kuhn--Tucker optimality conditions. As applications, we obtain a Frank--Wolfe type theorem which states that 
the optimal solution set of the problem is nonempty provided the objective function $f_0$ is convenient.
Finally, in the unconstrained case, we show that the optimal value of the problem is the smallest critical value of some polynomial.
All the results are presented in terms of the Newton polyhedra of the polynomials defining the problem.
\end{abstract}

\section{Introduction}

Optimality conditions form the foundations of mathematical programming both theoretically and computationally. There is a large literature on all aspects of optimality conditions. We refer the reader to the classical papers \cite{John1948,  Karush1939, Kuhn1951, Mangasarian1967} and to the comprehensive monographs \cite{Bertsekas1999, Clarke1990, Lasserre2015, Mordukhovich2006} with the references therein.

In this paper, we are interested in necessary optimality conditions to polynomial optimization problems {\em whose solution sets are empty.} More precisely, let $f_0, f_{1}, \ldots, f_{p} \colon \mathbb{R}^n \rightarrow  \mathbb{R}$ be polynomial functions and set
$$S := \{x \in \mathbb{R}^n \ : \ f_1(x) \le 0, \ldots, f_p(x) \le 0\}.$$
Assume that $S \ne \emptyset $ and the restriction of $f_0$ on $S$ is bounded from below. Consider the optimization problem
\begin{equation}\label{PT0}
f^* := \inf f_0(x) \quad \textrm{ such that } \quad x \in S. \tag{P}
\end{equation}
We first assume that $f_0$ attains its infimum on $S,$ i.e., $f^* = f_0(x^*)$ for some $x^* \in S.$ Thanks to the Fritz-John necessary optimality conditions \cite{John1948}, there exist nonnegative real numbers $\lambda^*_i, i = 0, 1, \ldots, p,$ not all zero, such that
\begin{eqnarray*}
&&  \lambda^*_0 \nabla  f_0({x}^*)  + \sum_{i = 1}^{p} \lambda^*_i \nabla f_i({x}^*) = 0, \\
&& \lambda^*_i f_i({x}^*) = 0 \quad \textrm{ for } i = 1, \ldots, p.
\end{eqnarray*}
Here and in the following, $\nabla  f_i(x)$ denotes the gradient vector of $f_i$ at $x.$ Furthermore, if the Mangasarian--Fromovitz constraint qualification \cite{Mangasarian1967} holds at $x^*$: there exists a vector $v \in \mathbb{R}^n$ such that 
$$\langle \nabla f_{i}(x^*) , v \rangle < 0 \quad \textrm{ for all \ $i \ge 1 $ \ with } \quad f_{i}(x^*) = 0,$$
then we may obtain the more informative optimality conditions due to Karush \cite{Karush1939}, Kuhn and Tucker \cite{Kuhn1951} where the real number $\lambda^*_0$ can be taken to be $1.$
Note that the Mangasarian--Fromovitz constraint qualification is generally satisfied; see the very recent paper by Bolte, Hochart, and Pauwels \cite{Bolte2017}.

We now assume that $f_0$ does not attain its optimal value $f^*$ on $S.$ As far as we know, there are no results which are similar to the two necessary optimality conditions mentioned above. The purpose of this paper is to fill this gap for polynomial programs. Indeed, under the assumption that the considered problem has the Mangasarian--Fromovitz property at infinity (see Definition~\ref{Definition2} in the next section), we will show that either $f^*= 0$ or there exist a nonempty set $J \subset \{1, \ldots, n\},$ a point $x^* \in \mathbb{R}^n,$ and scalars $\lambda^*_0, \lambda^*_1, \ldots, \lambda^*_p$ satisfying the following conditions:
\begin{enumerate}
\item [(i)]   $x^*_j = 0$ if, and only if, $j \not \in J;$
\item [(ii)]   $f^* = f_{0, \Delta_0}(x^*);$ 
\item [(iii)]  $\lambda^*_0 \nabla f_{0, \Delta_0}(x^*)   + \sum_{i = 1}^p \lambda^*_i \nabla f_{i, \Delta_i}(x^*)  = 0;$
\item [(iv)] $f_{i, \Delta_i}(x^*) \le 0$ and $\lambda^*_i f_{i, \Delta_i}(x^*)  = 0$ for $i = 1, \ldots, p;$
\item [(v)] the numbers $\lambda^*_i, i = 0, 1, \ldots, p,$ are nonnegative and not all zero;
\end{enumerate}
where $f_{i, \Delta_i}, i = 0 , \ldots, p,$ are polynomials which correspond to some faces $\Delta_i$ of the Newton polyhedra at infinity of $f_i.$  Moreover, if 
the following constraint qualification holds: there exists a vector $v \in \mathbb{R}^n$ such that 
$$\langle \nabla f_{i, \Delta_i}(x^*) , v \rangle < 0 \quad \textrm{ for all \ $i \ge 1$ \ with } \ f_{i, \Delta_i}(x^*) = 0 \ \textrm{ and } \ f_i|_{\mathbb{R}^J} \ne \mathrm{constant}.$$
then we can take $\lambda^*_0 = 1.$ 

In view of these results, we can say that $x^*$ is a ``minimizer at infinity'' of the problem~\eqref{PT0}.

Next, we study the existence of optimal solutions to polynomial optimization problems. It is well-known that for linear programming, a bounded feasible problem always has an optimal solution. This property is remarkable, and fails to hold for general nonlinear programs. Frank and Wolfe \cite{Frank1956} showed that when the objective function is quadratic and the feasible region is linear, the set of optimal solutions is nonempty provided the problem is bounded below. Many  other  authors  generalized  the  Frank--Wolfe  theorem to  broader  classes  of  functions. For example, Perold \cite{Perold1980} extended  the Frank--Wolfe theorem to a class of non-quadratic objective functions and linear constraints. Andronov,
Belousov, and Shironin \cite{Andronov1982} generalized  the Frank--Wolfe theorem to the case of a cubic polynomial objective function  under linear constraints. Luo and Zhang \cite{Luo1999} (see also \cite{Terlaky1985}) extended the Frank--Wolfe theorem to various classes of general convex or non-convex quadratic constraint systems. Belousov  and  Klatte  in  \cite{Belousov2002} (see also \cite{Bank1988, Belousov1977, Bertsekas2007, Obuchowska2006}) generalized the result on attainability to convex polynomial programs. Very recently, \DD inh, H\`a and the author \cite{Dinh2014-2} extended the Frank--Wolfe theorem for non-degenerate polynomial programs.

As a corollary of our theorem~\ref{Fritz-JohnTypeTheorem}, we readily establish the attainability of the infimum (assumed to be finite) of the problem~\eqref{PT0} when the objective function $f_0$ is convenient and the considered problem has the Mangasarian--Fromovitz property at infinity; this improves the main result in the paper~\cite{Dinh2014-2}. 

Finally, in the unconstrained case (i.e., $S = \mathbb{R}^n),$ we show that 
the optimal value of the problem~\eqref{PT0} is the smallest critical value of some polynomial.
This property is useful because, using Theorem~9 in the paper of Nie, Demmel, and Sturmfels~\cite{Nie2006}, it allows us to construct a sequence of semidefinite programming (SDP) relaxations whose optimal values converge monotonically, increasing to the optimal value of the original problem.

All the obtained results are presented in terms of the Newton polyhedra of $f_0$ and the polynomials defining $S.$ These results, together with those in  \cite{Dinh2014-2, Doat2014, HaHV2013, HaHV2017, Tomas2015}), show that many interesting properties in polynomial programming can be obtained from the geometry of Newton polyhedra.

The paper is structured as follows. 
In Section~\ref{Preliminary}, we recall some notations, definitions and preliminary facts which are used throughout this paper. 
Optimality conditions and the Frank--Wolfe theorem for polynomial programs are given in Section~\ref{Results}.
Further properties for the unconstrained case are given in Section~\ref{UnconstrainedCase}.

\section{Preliminaries} \label{Preliminary}

Throughout this paper, $\mathbb{R}^n$ denotes the Euclidean space of dimension $n.$ The corresponding inner product (resp., norm) in $\mathbb{R}^n$  is defined by $\langle x, y \rangle$ for any $x, y \in \mathbb{R}^n$ (resp., $\| x \| := \sqrt{\langle x, x \rangle}$ for any $x \in \mathbb{R}^n$). Given a nonempty set $J \subset \{1, \ldots, n\},$ we define
$$\mathbb{R}^J := \{x \in \mathbb{R}^n \ : \ x_j = 0, \textrm{ for all } j \not \in J\}.$$
We denote by $\mathbb{Z}_+$ the set of non-negative integer numbers. If $\kappa = (\kappa_1, \ldots, \kappa_n) \in \mathbb{Z}_+^n,$ we denote by $x^\kappa$ the monomial $x_1^{\kappa_1} \cdots x_n^{\kappa_n}.$

\subsection{Newton polyhedra and non-degeneracy conditions}

Let $f \colon \mathbb{R}^n \to \mathbb{R}$ be a polynomial function. Suppose that $f$ is written as $f = \sum_{\kappa} a_\kappa x^\kappa.$ Then
the support of $f,$ denoted by $\mathrm{supp}(f),$ is defined as the set of those $\kappa  \in \mathbb{Z}_+^n$ such that $a_\kappa \ne 0.$ 
The {\em Newton polyhedron (at infinity)}  of $f$, denoted by $\Gamma(f),$ is defined as the convex hull in $\mathbb{R}^n$ of the set $\mathrm{supp}(f).$\footnote{Note that we do not include the origin in the definition of the Newton polyhedron at infinity $\Gamma(f).$}
The polynomial $f$ is said to be {\em convenient} if $\Gamma(f)$ intersects each coordinate axis in a point different from the origin $0$ in $\mathbb{R}^n,$ that is, if for any $j \in \{1, \ldots, n\}$ there exists some $\kappa_j > 0$ such that $\kappa_j e^j \in \Gamma(f),$ where $\{e^1, \ldots, e^n\}$ denotes the canonical basis in $\mathbb{R}^n.$
For each (closed) face $\Delta$ of $\Gamma(f),$  we will denote by $f_\Delta$ the polynomial $\sum_{\kappa \in \Delta} a_\kappa x^\kappa;$ if $\Delta \cap \mathrm{supp}(f) = \emptyset$ we let $f_\Delta := 0.$

Given a nonzero vector $q \in \mathbb{R}^n,$ we define
\begin{eqnarray*}
d(q, \Gamma(f)) &:=& \min \{\langle q, \kappa \rangle \ : \ \kappa \in \Gamma(f)\}, \\
\Delta(q, \Gamma(f)) &:=& \{\kappa \in \Gamma(f) \ : \ \langle q, \kappa \rangle = d(q, \Gamma(f)) \}.
\end{eqnarray*}
By definition, for each nonzero vector $q \in \mathbb{R}^n,$ $\Delta(q, \Gamma(f)) $ is a closed face of $\Gamma(f).$ Conversely, if $\Delta$ is a closed face of $\Gamma(f)$ then there exists a nonzero vector\footnote{Since $\Gamma(f)$ is an integer polyhedron, we can assume that all the coordinates of $q$ are rational numbers.} $q \in \mathbb{R}^n$ such that $\Delta = \Delta(q, \Gamma(f)).$ The {\em Newton boundary (at infinity)} of $f$, denoted by $\Gamma_{\infty}(f),$ is defined as the union of all closed  faces
$\Delta(q, \Gamma(f))$ for some $q \in \mathbb{R}^n$ with $d(q, \Gamma(f))  < 0$ (and so $\min_{j = 1, \ldots, n}q_j < 0).$

Following N\'emethi and Zaharia \cite{Nemethi1990} we say that a closed face $\Delta$ of $\Gamma(f)$ is {\em  bad} if the following two conditions hold:
\begin{enumerate}
  \item [(i)] the affine subvariety of dimension = $\dim (\Delta)$ spanned by $\Delta$ contains the origin, and
  \item [(ii)] there exists a hyperplane $H \subset \mathbb{R}^n$ with equation $q_1\kappa_1 + q_2 \kappa_2 + \cdots + q_n \kappa_n = 0,$
where $\kappa_1, \kappa_2, \ldots, \kappa_n$ are the coordinates in $\mathbb{R}^n,$ such that:
\begin{enumerate}
  \item [(ii$_a$)] there exist $i$ and $j$ with $q_i \cdot q_j < 0,$ and
  \item [(ii$_b$)] $H \cap \Gamma(f) = \Delta.$
\end{enumerate}
\end{enumerate}
More geometrically, the condition (ii$_a$) says that the hyperplane $H$ intersects the interior of the positive octant of ${\Bbb R}^n.$ 

\begin{remark}\label{Remark21} {\rm
The following statements follow immediately from definitions:

(i) For each nonempty subset $J$ of $\{1, \ldots, n\},$ if the restriction of $f$ on $\mathbb{R}^J$ is not constant, then $\Gamma(f) \cap \mathbb{R}^J = \Gamma(f|_{\mathbb{R}^J}).$

(ii) If $f$ is convenient, then there is no bad face in $\Gamma(f)$ and, moreover,  $\Gamma(f) \cap \mathbb{R}^J = \Gamma(f|_{\mathbb{R}^J})$  for all nonempty subset $J$ of $\{1, \ldots, n\}.$

(iii) Let $\Delta := \Delta(q, \Gamma(f))$ for some nonzero vector $q := (q_1, \ldots, q_n) \in \mathbb{R}^n.$ By definition, $f_\Delta = \sum_{\kappa \in \Delta} a_\kappa x^\kappa$ is a weighted homogeneous polynomial of type $(q, d := d(q, \Gamma(f))),$ i.e., we have for all $t > 0$ and all $x \in \mathbb{R}^n,$
$$f_\Delta(t^{q_1} x_1,  \ldots, t^{q_n} x_n) = t^d f_\Delta(x_1, \ldots, x_n).$$
This implies the Euler relation
\begin{equation}\label{PT1}
\sum_{j = 1}^n q_j x_j \frac{\partial f_\Delta}{\partial x_j}(x) = d f_\Delta(x).
\end{equation}
In particular, if $d \ne 0$ and $\nabla f_\Delta(x) = 0,$ then $f_\Delta(x) = 0.$
}\end{remark}

The following notion  (see \cite{Kouchnirenko1976}) will play an important role in Section~\ref{UnconstrainedCase}. 
\begin{definition}\label{Definition1} {\rm
We say that $f$ is {\em (Kouchnirenko) nondegenerate at infinity} if, and only if, for all faces $\Delta \in  \Gamma_{\infty} (f),$ the system of equations
$$\frac{\partial f_{ \Delta }}{\partial x_1}(x)  =  \cdots =  \frac{\partial f_{ \Delta }}{\partial x_n}(x) = 0$$
has no solution in $(\mathbb{R} \setminus \{0\})^n.$
}\end{definition}

\begin{remark}{\rm
It is worth emphasizing that  the condition of non-degeneracy at infinity is a generic property in the sense that it holds in an open and dense semialgebraic set
of the entire space of input data (see, for example, \cite{HaHV2017}).
}\end{remark}

The following definition is inspired from the work of Mangasarian and Fromovitz~\cite{Mangasarian1967}.

\begin{definition}\label{Definition2} {\rm
Let $f_0, f_1, \ldots, f_p \colon \mathbb{R}^n \rightarrow \mathbb{R}$ be polynomial functions and set
$$S := \{x \in \mathbb{R}^n \ : \ f_1(x) \le 0, \ldots, f_p(x) \le 0\}.$$
We say that the {\em Mangasarian--Fromovitz property at infinity} ($\mathrm{(MF)}_\infty$ for short)  holds for the problem $\inf_{x \in S} f_0(x)$  if, and only if, 
for every nonempty set $J \subset \{1, \ldots, n\},$ for every vector $q \in \mathbb{R}^n,$ and for every $x \in \mathbb{R}^n$ satisfying the conditions
\begin{eqnarray*}
&\mathrm{(i)}& d(q, \Gamma(f_0)) < 0, \\
&\mathrm{(ii)}& \Delta_i := \Delta(q, \Gamma(f_i)) \subset \mathbb{R}^J \ \textrm{ for all } \ i \in  I := \{i \in \{0, 1, \ldots, p\} \ : \ f_i|_{\mathbb{R}^J} \ne \mathrm{constant}\}, \\
&\mathrm{(iii)}&  x_j = 0 \ \Longleftrightarrow \ j \not \in J,  \\
&\mathrm{(iv)}& f_{i, \Delta_i} (x) \le 0 \quad \textrm{ for } \quad  i = 1, \ldots, p,
\end{eqnarray*}
there exists a nonzero vector $v \in \mathbb{R}^n$ such that
\begin{eqnarray*}
\langle \nabla  f_{i, \Delta_i} (x), v \rangle  <  0 \quad \textrm{ for all \ $i \in I $ \ with } \ f_{i, \Delta_i} (x)  = 0.
\end{eqnarray*}
}\end{definition}

Note that the $\mathrm{(MF)}_\infty$ property in the above definition is not a constraint qualification since it involves the objective function $f_0.$ 

As shown in the next lemma, Definitions~\ref{Definition1}~and~\ref{Definition2} are equivalent in the unconstrained case.
\begin{lemma}\label{Lemma21}
Let $f_0 \colon \mathbb{R}^n \rightarrow \mathbb{R}$ be a polynomial function. Then $f_0$ is nondegenerate at infinity if, and only if, 
the problem $\inf_{x \in \mathbb{R}^n} f_0(x)$ has the $\mathrm{(MF)}_\infty$ property.
\end{lemma}

\begin{proof}
$\Rightarrow$. Take arbitrary a nonempty set $J \subset \{1, \ldots, n\},$ a nonzero vector $q \in \mathbb{R}^n,$ and a point $x^0 \in \mathbb{R}^n$ such that the following conditions hold:
\begin{enumerate}
\item[(a)] $d_0 := d(q, \Gamma(f_0)) < 0,$
\item[(b)]  $\Delta_0 := \Delta(q, \Gamma(f_0)) \subset \mathbb{R}^J, \quad f_0|_{\mathbb{R}^J} \ne \mathrm{constant}, $
\item[(c)] $x^0_j = 0 \ \Longleftrightarrow \ j \not \in J, $
\item[(d)] $f_{0, \Delta_0} (x^0) = 0.$
\end{enumerate}
By definition, $\Delta_0 \in \Gamma_\infty(f_0)$ and $f_{0, \Delta_0}$ does not depend on $x_j$ for $j \not \in J.$ The non-degeneracy assumption implies that there exists $j^* \in J$ such that $\frac{\partial f_{0, \Delta_0}}{\partial x_{j^*}}(x^0) \ne 0.$ Let $v := (v_1, \ldots, v_n) \in \mathbb{R}^n$ with
$$v_j := \begin{cases}
- \left(\frac{\partial f_{0, \Delta_0}}{\partial x_{j^*}}(x^0) \right)^{-1} & \textrm{ if } j  = j^*,\\
0 & \textrm{ otherwise.}
\end{cases}$$
We have $\langle \nabla  f_{0, \Delta_0} (x^0), v \rangle  = - 1 <  0,$ and so the problem $\inf_{x \in \mathbb{R}^n} f_0(x)$ has the $\mathrm{(MF)}_\infty$ property.

$\Leftarrow$. By contradiction, assume that there exist $\Delta_0 \in \Gamma_\infty(f_0)$ and $x^0 \in (\mathbb{R} \setminus \{0\})^n$ such that $\nabla  f_{0, \Delta_0} (x^0) = 0.$ 
By definition, there exists a vector $q \in \mathbb{R}^n$ such that $\Delta_0 = \Delta(q, \Gamma(f_0))$ and $d_0 := d(q, \Gamma(f_0)) < 0.$
Remark~\ref{Remark21}(iii) now leads to $f_{0, \Delta_0} (x^0) = 0.$ Let $J$ be the smallest subset of $\{1, \ldots, n\}$ such that the space $\mathbb{R}^J$ contains $\Delta_0.$ We have for all $\kappa \in \Delta_0 \subset \mathbb{R}^J,$
$$\sum_{j \in J} q_j \kappa_j \ = \ \sum_{j = 1}^n q_j \kappa_j \ = \ d_0 \ < \ 0.$$
It turns out that $\Gamma(f_0)  \cap \mathbb{R}^J$ is nonempty and different from $\{0\}.$ Consequently, $f_0|_{\mathbb{R}^J}$ is not constant. 
Let $y^0 := (y^0_1, \ldots, y^0_n) \in \mathbb{R}^n,$ where
$$y^0_j := \begin{cases}
x^0_j & \textrm{ if } j  \in J, \\
0 & \textrm{ otherwise.}
\end{cases}$$
Since $f_{0, \Delta_0}$ does not depend on $x_j$ for $j \not \in J,$ we get
$$f_{0, \Delta_0} (y^0) = f_{0, \Delta_0} (x^0) = 0 \quad \textrm{ and } \quad 
\nabla  f_{0, \Delta_0} (y^0) = \nabla  f_{0, \Delta_0} (x^0) = 0.$$
Combining these facts with the (MF)$_{\infty}$ property, we obtain the absurd conclusion:
\begin{eqnarray*}
0 &=& \langle \nabla  f_{0, \Delta_0} (y^0), v \rangle  \ < \   0
\end{eqnarray*}
for some vector $v \in \mathbb{R}^n.$ 
\end{proof}

\subsection{Semi-algebraic geometry}
This subsection contains some background material on semi-algebraic geometry and preliminary results which will be used later. We give only concise definitions and results that will be needed in the paper. For more detailed information on the subject, see, for example, \cite{Bochnak1998} and \cite[Chapter~1]{HaHV2017}.

\begin{definition}{\rm
\begin{enumerate}
  \item[(i)] A subset of $\mathbb{R}^n$ is called {\em semi-algebraic} if it is a finite union of sets of the form
$$\{x \in \mathbb{R}^n \ | \ f_i(x) = 0, i = 1, \ldots, k; f_i(x) > 0, i = k + 1, \ldots, p\}$$
where all $f_{i}$ are polynomials.
 \item[(ii)] Let $A \subset \mathbb{R}^n$ and $B \subset \mathbb{R}^m$ be semi-algebraic sets. A map $F \colon A \to B$ is said to be {\em semi-algebraic} if its graph
$$\{(x, y) \in A \times B \ | \ y = F(x)\}$$
is a semi-algebraic subset in $\mathbb{R}^n\times \mathbb{R}^m.$
\end{enumerate}
}\end{definition}

The class of semi-algebraic sets is closed under taking finite intersections, finite unions and complements; a Cartesian product of semi-algebraic sets is a semi-algebraic set. Moreover, a major fact concerning the class of semi-algebraic sets is its stability under linear projections; in particular, the closure and the interior of a semi-algebraic set are semi-algebraic sets.

In the sequel, we will need the following useful results (see, for example, \cite[Chapter~1]{HaHV2017}).

\begin{lemma}[Curve Selection Lemma at infinity]\label{CurveSelectionLemma}
Let $A\subset \mathbb{R}^n$ be a semi-algebraic set, and let $f := (f_1, \ldots,f_p) \colon  \mathbb{R}^n \to \mathbb{R}^p$ be a semi-algebraic map. Assume that there exists a sequence $\{x^\ell\}$ such that $x^\ell \in A$, $\lim_{l \to \infty} \| x^\ell  \| = \infty$ and $\lim_{l \to \infty} f(x^\ell)  = y \in(\overline{\mathbb{R}})^p,$ where $\overline{\mathbb{R}} := \mathbb{R} \cup \{\pm \infty\}.$ Then there exists a smooth semi-algebraic curve $\varphi \colon (0, \epsilon)\to \mathbb{R}^n$ such that $\varphi(t) \in A$ for all $t \in (0, \epsilon), \lim_{t \to 0} \|\varphi(t)\| = \infty,$ and $\lim_{t \to 0} f(\varphi(t)) = y.$
\end{lemma}

\begin{lemma}[Growth Dichotomy Lemma] \label{GrowthDichotomyLemma}
Let $f \colon (0, \epsilon) \rightarrow \mathbb{R}$ be a semi-algebraic function with $f(t) \ne 0$ for all $t \in (0, \epsilon).$ Then
there exist constants $c \ne 0$ and $q \in \mathbb{Q}$ such that $f(t) = ct^q + o(t^q)$ as $t \to 0^+.$
\end{lemma}

\begin{lemma}[Monotonicity Lemma] \label{MonotonicityLemma}
Let $a < b$ in $\mathbb{R}.$ If $f \colon [a, b] \rightarrow \mathbb{R}$ is a semi-algebraic function, then there is a partition $a =: t_1 < \cdots < t_{N} := b$ of $[a, b]$ such that $f|_{(t_l, t_{l + 1})}$ is $C^1,$ and either constant or strictly monotone, for $l \in \{1, \ldots, N - 1\}.$
\end{lemma}

\section{The constrained case} \label{Results}

From now on we let $f_0, f_{1}, \ldots, f_{p} \colon \mathbb{R}^n \rightarrow  \mathbb{R}$ be nonconstant polynomial functions and set
$$S := \{x \in \mathbb{R}^n \ : \ f_1(x) \le 0, \ldots, f_p(x) \le 0\}.$$
We will assume that $S \ne \emptyset$ and $f_0$ is bounded from below on $S.$ Consider the problem~\eqref{PT0} formulated in the introduction section:
\begin{equation*}
f^* := \inf f_0(x) \quad \textrm{ such that } \quad x \in S. \tag{P}
\end{equation*}
The main result of this paper is the following theorem, which is a version at infinity of the Fritz-John and Karush--Kuhn--Tucker optimality conditions.

\begin{theorem}[Optimality conditions for minimizers at infinity] \label{Fritz-JohnTypeTheorem}
Assume that the $\mathrm{(MF)}_\infty$ property holds for the problem~\eqref{PT0}. If $f_0$ does not attain its infimum $f^*$ on $S,$ then 
either $f^* = 0$ or there exist a nonempty set $J \subset \{1, \ldots, n\},$ a vector $q \in \mathbb{R}^n$ with $\min_{j \in J} q_j < 0,$ a point $x^* \in \mathbb{R}^n,$ and scalars $\lambda^*_0, \lambda^*_1, \ldots, \lambda^*_p$ such that the following conditions hold:
\begin{enumerate}
\item [(i)]   $x^*_j = 0$ if, and only if, $j \not \in J;$
\item [(ii)]   $f^* = f_{0, \Delta_0}(x^*),$ $\Delta_0 := \Delta(q, \Gamma(f_0))  \subset \mathbb{R}^J,$ and $d(q, \Gamma(f_0)) = 0;$
\item [(iii)]  $\lambda^*_0 \nabla f_{0, \Delta_0}(x^*)   + \sum_{i = 1}^p \lambda^*_i \nabla f_{i, \Delta_i}(x^*)  = 0;$
\item [(iv)] $f_{i, \Delta_i}(x^*) \le 0$ and $\lambda^*_i f_{i, \Delta_i}(x^*)  = 0$ for $i = 1, \ldots, p;$
\item [(v)] the numbers $\lambda^*_i, i = 0, 1, \ldots, p,$ are nonnegative and not all zero;
\end{enumerate}
where $\Delta_i := \Delta(q, \Gamma(f_i )).$ Moreover, we can take $\lambda^*_0 = 1$ provided the following constraint qualification holds: there exists a nonzero vector $v \in \mathbb{R}^n$ such that 
$$\langle \nabla f_{i, \Delta_i}(x^*) , v \rangle < 0 \quad \textrm{ for all \ $i \ge 1$ \ with } \ f_{i, \Delta_i}(x^*) = 0 \ \textrm{ and } \ f_i|_{\mathbb{R}^J} \ne \mathrm{constant}.$$
\end{theorem}

\begin{proof}
Since $f_0$ does not attain its infimum $f^*$ on $S,$ there exists a sequence $\{a^{\ell}\}_{\ell \ge 1} \subset S$ such that
$$\lim_{\ell \to \infty} \|a^\ell\| = + \infty \quad \textrm{ and } \quad \lim_{\ell \to \infty} f_0(a^\ell) = f^*.$$
For each $\ell \ge 1,$ we consider the problem
\begin{eqnarray*}
&& \textrm{minimize} \qquad f_0(x) \\
&& \textrm{subject to } \quad f_i(x) \le 0, \ i = 1, \ldots, p, \quad \|x\|^2 = \|a^\ell\|^2.
\end{eqnarray*}
Since the objective function $f_0$ is continuous and the constraint set is compact, by the Weierstrass theorem, an optimal solution $x^\ell \in S$ of the problem exists. We have
$$\|x^\ell\|^2 = \|a^\ell\|^2 \quad \textrm{ and } \quad f^* \le f_0(x^\ell) \le f_0(a^\ell).$$
Hence,
$$\lim_{\ell \to \infty} \|x^\ell\| = + \infty \quad \textrm{ and } \quad \lim_{\ell \to \infty} f_0(x^\ell) = f^*.$$
Since the number of all subsets of the set $\{1, \ldots, n\}$ is finite, by passing to a subsequence if necessary,  we may assume that $\{j \in \{1, \ldots, n\} \ : \ x^\ell_j \ne 0 \} = J$  for all  $\ell $ and for some nonempty set $J \subset \{1, \ldots, n\}.$ 

Let $I_1$ be the (possibly empty) set of all indices $i \in \{1, \ldots, p\}$ such that the restriction of $f_i$ on $\mathbb{R}^J = \{x \in \mathbb{R}^n \ : \ x_j = 0, j \not\in J\}$ is not constant. We have for all $i \in \{1, \ldots, p\} \setminus I_1,$
\begin{eqnarray} \label{PT2}
f_i(x) & = & f_i(x^\ell)  \ \le \ 0 \quad \textrm{ for all } x \in \mathbb{R}^J \textrm{ and all } \ \ell \ge 1.
\end{eqnarray}
Consequently, $x^\ell$ is also an optimal solution of the problem
\begin{eqnarray*}
&& \textrm{minimize} \qquad f_0(x) \\
&& \textrm{subject to } \quad f_i(x) \le 0, i \in I_1, \quad h_k(x) = 0, k \not\in J, \quad \|x\|^2 = \|a^\ell\|^2,
\end{eqnarray*}
here and in the following we let $h_k(x) := x_k.$ The Fritz-John optimality conditions (see, for example, \cite{Bertsekas1999}) imply that there exist some real numbers $\lambda^\ell_i$ for $i \in I_1 \cup \{0\},$ $\nu^\ell_k$ for $k \not \in J,$ and $\mu^\ell$ such that the following relations hold:
\begin{eqnarray*}
&& \sum_{i \in I_1 \cup \{0\}} \lambda_i^\ell \nabla f_i(x^\ell) +  \sum_{k \not \in J}  \nu^\ell_k \nabla h_k(x^\ell) + \mu^\ell x^\ell = 0,\\
&&  \lambda^\ell_0 \ge 0, \ \lambda^\ell_i \ge 0 \ \textrm{ and } \ \lambda^\ell_i f_{i}(x^\ell)  = 0 \quad \textrm{ for } i \in I_1,\\
&&  \|(\lambda_i^\ell, \nu^\ell_k, \mu^\ell)_{i \in I_1 \cup \{0\}, k \not \in J} \| = 1.
\end{eqnarray*}

For simplicity, we write  $\lambda := (\lambda_i)_{i \in I_1 \cup \{0\}} \in \mathbb{R}^{\#I_1 + 1}$ and $\nu := (\nu_k)_{k \not \in J} \in \mathbb{R}^{n - \# J}.$ Let
\begin{eqnarray*}
\mathscr{A} := \big \{ (x, \lambda, \nu, \mu) & \in & \mathbb{R}^n \times \mathbb{R}^{\#I_1 + 1} \times \mathbb{R}^{n - \# J} \times \mathbb{R}
\  :  \ f_i(x) \le 0 \textrm{ for }  i \in I_1, \\
&& x_j \ne 0 \textrm{ for } j \in J, \ h_k(x) = 0 \textrm{ for } k \not \in J,\\
&& \sum_{i \in I_1 \cup \{0\}} \lambda_i \nabla f_i(x) + \sum_{k \not \in J}  \nu_k \nabla h_k(x) + \mu x = 0, \\
&& \lambda_i f_i(x) = 0 \textrm{ for } i \in I_1, \ \lambda_i \ge 0 \textrm{ for } i \in I_1 \cup \{0\}, \ \|(\lambda, \nu, \mu)\| = 1 \big \}.
\end{eqnarray*}
Then $\mathscr{A}$ is a semi-algebraic set and the sequence $(x^\ell, \lambda^\ell, \nu^\ell, \mu^\ell) \in \mathscr{A}$ tends to infinity as $\ell \to \infty.$ Applying Lemma ~\ref{CurveSelectionLemma} to the semi-algebraic function $\mathscr{A} \rightarrow \mathbb{R}, (x, \lambda, \nu, \mu)  \mapsto f_0(x),$ 
we get a smooth semi-algebraic curve 
$$(\varphi, \lambda, \nu, \mu) \colon (0, \epsilon) \rightarrow \mathbb{R}^n \times \mathbb{R}^{\#I_1 + 1} \times \mathbb{R}^{n - \# J} \times \mathbb{R}, \quad t \mapsto (\varphi(t), \lambda(t), \nu(t), \mu(t)),$$ 
satisfying the following conditions
\begin{enumerate}
\item [{(c1)}] $\lim_{t \to 0^+}  \|\varphi(t)\|  = +\infty;$ 
\item [{(c2)}] $\lim_{t \to 0^+}  f_0(\varphi(t)) = f^*;$
\item [{(c3)}] $f_i(\varphi(t)) \le 0$ for $i \in I_1;$
\item [{(c4)}] $J = \{ j \in \{1, \ldots, n\} \ : \ \varphi_j(t) \not \equiv 0\};$
\item [{(c5)}] $h_k(\varphi(t)) \equiv  0$ for $k \not \in J;$
\item [{(c6)}] $\sum_{i \in I_1 \cup \{0\}}  \lambda_i(t) \nabla f_i(\varphi(t)) + 
\sum_{k \not \in J}  \nu_k(t) \nabla h_k(\varphi(t)) + \mu(t) \varphi(t) \equiv 0;$
\item [{(c7)}]$\lambda_i(t) f_i(\varphi(t)) \equiv 0,$ for $i \in I_1;$
\item [{(c8)}]$\lambda_i(t) \ge 0$ for $i \in I_1 \cup \{0\};$
\item [{(c9)}] $\|(\lambda(t), \nu(t),  \mu(t))\| \equiv 1.$
\end{enumerate}

Since the (smooth) functions $\lambda_i$ and $f_i \circ \varphi$ are semi-algebraic, by shrinking $\epsilon$ if necessary, we can assume, that these functions are either constant or strictly monotone (see Lemma~\ref{MonotonicityLemma}). Then, by Condition~(c7), one can see for all $i \in I_1$ that either $\lambda_i(t) \equiv 0$ or $f_i\circ \varphi(t) \equiv  0.$ Consequently, we obtain
$$\lambda_i(t) \frac{d}{ dt }(f_i \circ \varphi)(t)  \equiv 0 \quad \textrm{ for } \quad i \in I_1.$$

Let $I_2 := \{i \in I_1 \ : \ f_i \circ \varphi (t) \equiv 0\}.$ We have $\lambda_i(t) \equiv 0$ for $i \in I_1 \setminus I_2$ because of Condition~(c7). It follows from Condition~(c6) that
\begin{eqnarray*}
\frac{\mu (t)}{2} \frac{d \|\varphi(t)\|^2}{dt}
&=& \mu (t) \left \langle \varphi(t), \frac{d \varphi(t)}{dt} \right \rangle \\
&=& - \sum_{i \in I_2 \cup \{0\}} \lambda_i(t) \left \langle \nabla  f_i(\varphi(t)), \frac{d \varphi(t)}{dt} \right \rangle 
- \sum_{k \not \in J} \nu_k(t) \left \langle \nabla  h_k(\varphi(t)), \frac{d \varphi(t)}{dt} \right \rangle \\
&=& - \sum_{i \in I_2 \cup \{0\} } \lambda_i(t) \frac{d}{dt}(f_i \circ \varphi)(t) - \sum_{k \not \in J} \nu_k(t) \frac{d}{dt}(h_k \circ \varphi)(t) \\
&=& - \lambda_0(t) \frac{d}{dt}(f_0 \circ \varphi)(t) - \sum_{k \not \in J} \nu_k(t) \frac{d}{dt}(h_k \circ \varphi)(t).
\end{eqnarray*}
From Condition (c5) one has $\frac{d}{dt}(h_k \circ \varphi)(t) = 0$ for all $k \not \in J$ and hence
\begin{eqnarray}\label{PT3}
\frac{\mu (t)}{2} \frac{d \|\varphi(t)\|^2}{dt} &=& - \lambda_0(t) \frac{d}{dt}(f_0 \circ \varphi)(t).
\end{eqnarray}
This, together with Condition~(c1), implies that if $\|(\lambda_i(t))_{i \in I_1 \cup\{0\}}\|  = 0$ then $\mu(t) = \lambda_0(t) = 0,$ and hence, by Condition (c6), 
$$\sum_{k \not \in J}  \nu_k(t) \nabla h_k(\varphi(t)) =  0.$$
Combining this with the definition of $h_k$ we see that $\nu_k(t) = 0$ for all $k \not \in J,$ which contradicts Condition~(c9).
Thus, $\|(\lambda_i(t))_{i \in I_1 \cup\{0\}}\|  \ne 0,$ and so, after a scaling, we can assume that 
\begin{eqnarray}\label{PT4}
\|(\lambda_i(t))_{i \in I_1 \cup \{0\}}\| & = & 1 \quad \textrm{ for all } \quad t \in (0, \epsilon).
\end{eqnarray}

From Condition~(c4) we have $\varphi_j \not\equiv 0$ for all $j \in J.$ By Lemma~\ref{GrowthDichotomyLemma}, for each $j \in J,$ we can expand the coordinate $\varphi_j$ as follows
$$\varphi_j(t) =  {x}^*_j t^{q_j} + \textrm{ higher order terms in } t,$$
where ${x}^*_j \ne 0$ and $q_j \in \mathbb{Q}.$ From Condition (c1), we get $\min_{j \in J} q_j  < 0.$

Let $q_j := M$ for $j \not \in J$ with
\begin{eqnarray*}
M & \gg & \max_{i = 0, 1, \ldots, p}\left  \{\sum_{j \in J} q_j \kappa_j \ : \ \kappa \in \Gamma(f_i) \right\}.
\end{eqnarray*}
For each $i \in \{0, 1, \ldots, p\},$ let $d_i$ be the minimal value of the linear function $\sum_{j = 1}^n q_j \kappa_j$ on $\Gamma(f_i)$ and let $\Delta_i$ be the maximal face of $\Gamma(f_i)$ (maximal with respect to the inclusion of faces)  where the linear function takes this value, i.e.,
\begin{eqnarray*}
d_i &:=& d(q, \Gamma(f_i)) \quad \textrm{ and } \quad \Delta_i \ := \ \Delta(q, \Gamma(f_i)).
\end{eqnarray*}

Recall that $\mathbb{R}^J := \{\kappa := (\kappa_1, \ldots, \kappa_n) \in \mathbb{R}^n \ : \ \kappa_j = 0 \textrm{ for } j \not \in J \}.$ 
Take any $i \in I_1 \cup \{0\}.$  Then the restriction of $f_i$ on $\mathbb{R}^J$ is not constant, and so 
$\Gamma(f_i) \cap {\mathbb{R}^J} = \Gamma(f_i|_{\mathbb{R}^J}) $ is nonempty and different from $\{0\}.$ Furthermore,  by definition of the vector $q,$ one has 
\begin{eqnarray} \label{PT5}
d_i &=& d(q, \Gamma(f_i|_{\mathbb{R}^J})) \quad \textrm{ and } \quad \Delta_i \ = \ \Delta(q, \Gamma(f_i|_{\mathbb{R}^J})) 
\ \subset \ {\mathbb{R}^J}.
\end{eqnarray}
Consequently, for each $j \not \in J,$ the polynomial $f_{i,\Delta_i}$ does not depend on the variable $x_j.$
Now suppose that $f_i$ is written as $f_i(x) = \sum_{\kappa} a_{i, \kappa} x^\kappa.$ Then
\begin{eqnarray*}
f_i(\varphi(t)) &=& \sum_{\kappa\in\Gamma(f_i) \cap \mathbb{R}^J} a_{i, \kappa} (\varphi(t))^\kappa\\
&=&\sum_{\kappa\in\Gamma(f_i) \cap \mathbb{R}^J}a_{i, \kappa} (\varphi_1(t))^{\kappa_1}\ldots(\varphi_n(t))^{\kappa_n}\\
&=&\sum_{\kappa\in\Gamma(f_i) \cap \mathbb{R}^J}\left(a_{i, \kappa} ({x}^*_1 t^{q_1})^{\kappa_1}\ldots({x}^*_nt^{q_n})^{\kappa_n}+\textrm{ higher order terms in } t\right)\\
&=&\sum_{\kappa\in\Gamma(f_i) \cap \mathbb{R}^J}\left(a_{i, \kappa} ({x}^*)^\kappa t^{\sum_{j \in J}q_j \kappa_j}+\textrm{ higher order terms in } t\right)\\
&=&\sum_{\kappa\in\Delta_i}a_{i, \kappa} ({x}^*)^\kappa t^{d_i}+\textrm{ higher order terms in } t,
\end{eqnarray*}
where ${x}^* := ({x}^*_1, \ldots, {x}^*_n)$ with ${x}^*_j := 0$ for $j \not \in J.$ By definition, $f_{i, \Delta_i} (x) = \sum_{\kappa \in \Delta_i} a_{i, \kappa} x^\kappa.$ Hence
\begin{eqnarray}\label{PT6}
f_i(\varphi(t)) &=& f_{i, \Delta_i}({x}^*)t^{d_i} + \textrm{ higher order terms in } t.
\end{eqnarray}
If $d_0 > 0,$ then it follows from Condition~(c2) that $f^* = 0$ and the theorem is proved.
So, in the rest of the proof, we assume that $d_0 \le 0.$ 
Observe that if $d_0 < 0$ then $f_{0, \Delta_0}({x}^*) = 0.$ Therefore,
\begin{eqnarray}\label{PT7}
d_0 \le 0 \quad \textrm{ and } \quad  d_0 f_{0, \Delta_0}({x}^*) = 0.
\end{eqnarray}
Furthermore, it follows from (c3), \eqref{PT6} and the definition of the sets $I_1, I_2$ that
\begin{equation}\label{PT8}
f_{i, \Delta_i}({x}^*) \le 0 \quad \textrm{ for all } \quad i \in I_1 \quad \textrm{ and } \quad 
f_{i, \Delta_i}({x}^*) = 0 \quad \textrm{ for all } \quad i \in I_2.
\end{equation}

Let $I_3 := \{i \in I_2 \cup \{0\} \ : \ \lambda_i \not \equiv 0 \}.$ Since $\lambda_i \equiv 0$ for all $i \in I_1 \setminus I_2,$ we obtain from \eqref{PT4} that $I_3 \ne \emptyset.$ For $i \in I_3,$ expand the coordinate $\lambda_i$ in terms of the parameter  (cf.~Lemma~\ref{GrowthDichotomyLemma}) as follows
$$\lambda_i(t) =  \lambda_i^0 t^{\theta_i} + \textrm{ higher order terms in } t,$$
where $\lambda_i^0 \ne 0$ and $\theta_i  \in \mathbb{Q}.$ By Condition~(c8), then $\lambda_i^0 > 0.$ Furthermore, from \eqref{PT4} one has $\theta_i \ge 0$ for all $i \in I_3$ with the equality occurring for some $i \in I_3.$ 

For $i \in I_3$ and $j \in J$, by some similar calculations as with $f_i(\varphi(t))$, we have 
\begin{eqnarray*}
\frac{\partial f_i}{\partial x_j}(\varphi(t))
&=& \frac{\partial f_{i, \Delta_i}}{\partial x_j}({x}^*)t^{d_i - q_j}  + \textrm{ higher order terms in } t.
\end{eqnarray*}
Since $h_k(x) = x_k$ for all $k \not \in J,$ it holds that
\begin{eqnarray*}
\frac{\partial h_k}{\partial x_j}(\varphi(t)) &\equiv& 0 \quad \textrm{ for all } \quad j \in J.
\end{eqnarray*}
Consequently, we have for all $j \in J,$
\begin{eqnarray}\label{PT9}
\sum_{i \in I_3} \lambda_i(t) \frac{\partial f_i}{\partial x_j}(\varphi(t)) &+&
\sum_{k \not \in J} \nu_k(t) \frac{\partial h_k}{\partial x_j}(\varphi(t))
\ = \  \sum_{i \in I_3} \lambda_i(t) \frac{\partial f_i}{\partial x_j}(\varphi(t)) \nonumber\\
&=& \sum_{i \in I_3} \left( \lambda_i^0  \frac{\partial f_{i, \Delta_i}}{\partial x_j}({x}^*)t^{d_i + \theta_i - q_j}  + \textrm{ higher order terms in } t \right) \nonumber \\
&=& \left( \sum_{i \in I_4} \lambda_i^0  \frac{\partial f_{i, \Delta_i}}{\partial x_j}({x}^*) \right) t^{\ell - q_j}  + \textrm{ higher order terms in } t, 
\end{eqnarray}
where $\ell := \min_{i \in I_3} (d_i + \theta_i)$ and $I_4 := \{i \in I_3 \ : \ d_i + \theta_i = \ell\} \ne \emptyset.$ 

\begin{Claim} 
We have
\begin{eqnarray}\label{PT10}
\sum_{i \in I_4} \lambda_i^0  \frac{\partial f_{i, \Delta_i}}{\partial x_j}({x}^*)  &=& 0 \quad \textrm{ for all } \quad j = 1, \ldots, n.
\end{eqnarray}
\end{Claim}

\begin{proof}
Indeed, for each $j \not \in J,$ the polynomial $f_{i,\Delta_i}$ does not depend on the variable $x_j$, so $\frac{\partial f_{i, \Delta_i}}{\partial x_j}\equiv 0.$ Consequently, 
\begin{eqnarray*}
\sum_{i \in I_4} \lambda_i^0  \frac{\partial f_{i, \Delta_i}}{\partial x_j}({x}^*)= 0 \quad \textrm{ for all } \quad j \not \in J.
\end{eqnarray*}

If $\mu(t) \equiv 0$ then Condition~(c6) and \eqref{PT9} give
\begin{eqnarray*}
\sum_{i \in I_4} \lambda_i^0  \frac{\partial f_{i, \Delta_i}}{\partial x_j}({x}^*)= 0 \quad \textrm{ for all } \quad j \in J,
\end{eqnarray*}
and there is nothing to prove. So assume that $\mu(t) \not \equiv 0.$ By Lemma~\ref{GrowthDichotomyLemma}, we may write
\begin{eqnarray*}
\mu(t) & = & \mu^0 t^{\delta} + \textrm{higher order terms in } t,
\end{eqnarray*}
where $\mu^0 \ne 0$ and $\delta \in \mathbb{Q}.$ Let $J' := \{j \in J \ : \ \ell - q_j = \delta + q_j\}.$ 
Assume that $J' \ne \emptyset.$ We have from (c6) and \eqref{PT9} that
$\ell - q_j \le \delta + q_j$ for all $j \in J$ and that
\begin{eqnarray*}
\sum_{i \in I_4} \lambda_i^0  \frac{\partial f_{i, \Delta_i}}{\partial x_j}({x}^*) = 
\begin{cases}
-\mu^0 {x}^*_j  & \textrm{ if  } j \in J',\\
0 & \textrm{ otherwise.}
\end{cases}
\end{eqnarray*}
Hence
\begin{eqnarray*}
\sum_{j = 1}^n \left ( \sum_{i \in I_4} \lambda_i^0  \frac{\partial f_{i, \Delta_i}}{\partial x_j}({x}^*) \right) q_j {x}^*_j 
&=& \sum_{j \in J'} \left ( \sum_{i \in I_4} \lambda_i^0  \frac{\partial f_{i, \Delta_i}}{\partial x_j}({x}^*) \right) q_j {x}^*_j \\
&=& \sum_{j \in J'} - q_j \mu^0 ({x}^*_j)^2  \\
&=& - \frac{\ell - \delta}{2} \mu^0 \sum_{j \in J'}({x}^*_j)^2.
\end{eqnarray*}

On the other hand, $f_{i, \Delta_i}$ is a weighted homogeneous polynomial of type $(q, d_i).$ Thus, from the Euler relation~\eqref{PT1} we obtain for all $i \in I_4,$
\begin{eqnarray*}
\sum_{j = 1}^n q_j {x}^*_j \frac{\partial f_{i, \Delta_i}}{\partial x_j}({x}^*)  = d_i f_{_i, \Delta_i}({x}^*) \ = \ 0,
\end{eqnarray*}
where the last equality follows from \eqref{PT7} and \eqref{PT8}. But $\ell - \delta \ne 0$ because $\ell - q_j \le \delta + q_j$ for all $j \in J$ and $\min_{j \in J} q_j < 0.$ Hence, we obtain the absurd equality
\begin{eqnarray*}
0 
&=& \sum_{i \in I_4} \left ( \sum_{j = 1}^n q_j {x}^*_j \frac{\partial f_{i, \Delta_i}}{\partial x_j}({x}^*) \right) \lambda^0_i 
\ = \ \sum_{j = 1}^n  \left ( \sum_{i \in I_4}  \lambda^0_i \frac{\partial f_{i, \Delta_i}}{\partial x_j}({x}^*) \right) q_j {x}^*_j  \\
&=& - \frac{\ell - \delta}{2} \mu^0 \sum_{j \in J'}({x}^*_j)^2 \ \ne \ 0.
\end{eqnarray*}
Therefore, $J' = \emptyset.$ The claim is proved.
\end{proof}

\noindent
{\em Proof of Theorem~\ref{Fritz-JohnTypeTheorem} (continued).} 
Let
$\lambda^* := (\lambda^*_0, \lambda^*_1, \ldots, \lambda^*_p) \in \mathbb{R}^{p + 1},$ 
where
$$\lambda^*_i := 
\begin{cases}
\lambda^0_i & \textrm{ if } i \in I_4, \\
0 & \textrm{ otherwise.}
\end{cases}$$
Then the real numbers $\lambda^*_i, i = 0, 1, \ldots, p,$ are nonnegative and not all zero. This proves the statements~(i), (iii) and (v) when combined with~\eqref{PT10}.

Take any $i \in \{1, \ldots, p\} \setminus I_1.$ The restriction of $f_i$ on $\mathbb{R}^J$ is constant. Combining this with~\eqref{PT2}, we get
$$f_{i, \Delta_i}(x^*) = f_i(x^*) = f_i(x^\ell) \le 0.$$ 
Hence, the statement (iv) follows from \eqref{PT8}.

We next show that $d_0 = 0.$ If it is not the case, then the $\mathrm{(MF)}_\infty$ property shows the existence of a vector $v \in \mathbb{R}^n$ satisfying
\begin{eqnarray*}
\langle \nabla  f_{i, \Delta_i} (x^*), v \rangle  <  0 \quad \textrm{ for all \ $i \in I_1 \cup \{0\}$ \ with } \ f_{i, \Delta_i} (x^*)  = 0.
\end{eqnarray*}
(Note that $I_1 \cup \{0\}$ is the set of indices $i \ge 0$ for which the restriction of $f_i$ on ${\mathbb{R}^J}$ is not constant.)
These inequalities, together with the proved statements (iii)-(v), give a contradiction:
\begin{eqnarray*}
0 
&=& \lambda^*_0 \langle \nabla f_{0, \Delta_0}(x^*) , v \rangle  + \sum_{i = 1}^p \lambda^*_i \langle \nabla f_{i, \Delta_i}(x^*), v \rangle \\
&=& \sum_{i \in I_4} \lambda^*_i \langle \nabla f_{i, \Delta_i}(x^*), v \rangle \ < \ 0.
\end{eqnarray*}
Therefore, $d(q, \Gamma(f_0)) = d_0 = 0.$ Furthermore, combining \eqref{PT6} with Condition~(c2), we deduce that $f^* = f_{0, \Delta_0}(x^*) ,$ and the statement~(ii) follows.

Finally, let $v \in \mathbb{R}^n$ be a vector such that 
$$\langle \nabla f_{i, \Delta_i}(x^*) , v \rangle < 0 \quad \textrm{ for all \ $i \in I_1$ \ with } \quad  f_{i, \Delta_i}(x^*) = 0.$$
Since the multipliers $\lambda^*_i$ can be normalized by multiplication with a positive scalar, it is sufficient to show that $\lambda^*_0 > 0.$

To the contrary, assume that $\lambda^*_0 = 0,$ so that
\begin{eqnarray*}
\sum_{i = 1}^p \lambda^*_i \nabla f_{i, \Delta_i}(x^*)  &=& 0.
\end{eqnarray*}
This leads to an absurd situation
\begin{eqnarray*}
0 &=& \sum_{i = 1}^p \lambda^*_i \langle \nabla f_{i, \Delta_i}(x^*) , v \rangle 
\ = \ \sum_{i \in I_4, \ i \ge 1} \lambda^*_i \langle \nabla f_{i, \Delta_i}(x^*) , v \rangle  \ < \ 0.
\end{eqnarray*}
Therefore, we must have $\lambda^*_0 > 0,$ which completes the proof.
\end{proof}

\begin{remark}{\rm
(i) We do not know whether $x^*$ is a minimizer or not of the polynomial optimization problem 
\begin{equation*}
\inf f_{0, \Delta_0}(x) \quad \textrm{ subject to } \quad f_{i, \Delta_i}(x) \le 0 \ \textrm{ for } \ i = 1, \ldots, p.
\end{equation*}

(ii) Since the restriction of $f_0$ on $S$ does not attain its infimum $f^*,$ Condition~(c2) shows that the function $(0, \epsilon) \rightarrow \mathbb{R}, t \mapsto (f_0 \circ \varphi)(t),$ is strictly decreasing
(after perhaps shrinking $\epsilon).$ Now, by Lemma~\ref{GrowthDichotomyLemma}, we may write
\begin{eqnarray*}
f_0(\varphi(t)) & = & f^* + c t^{\alpha} + \textrm{higher order terms in } t,
\end{eqnarray*}
where $c > 0$ and $\alpha > 0.$ 

On the other hand,  we deduce from \eqref{PT3} and Condition~(c6) that
\begin{eqnarray*}
\left | \lambda_0(t)  \frac{d}{dt}(f_0 \circ \varphi)(t)\right | 
&=& 
\frac{\left \| \displaystyle \sum_{i \in I_1 \cup \{0\}}  \lambda_i(t) \nabla f_i(\varphi(t)) + 
\sum_{k \not \in J}  \nu_k(t) \nabla h_k(\varphi(t)) \right  \|}{2\| \varphi(t)\|} \left | \frac{d
\|\varphi(t)\|^2}{d t} \right |.
\end{eqnarray*}
Note by (c9) that $| \lambda_0(t) | \le 1$ for all $t \in (0, \epsilon).$ Then a simple calculation shows that
\begin{eqnarray*}
\|\varphi(t)\| \left\|\sum_{i \in I_1 \cup \{0\}}  \lambda_i(t) \nabla f_i(\varphi(t)) + \sum_{k \not \in J}  \nu_k(t) \nabla h_k(\varphi(t)) \right\| &=& c' t^{\alpha}   + \textrm{higher order terms in } t,
\end{eqnarray*}
for some constant $c' \ge 0.$ Since $\alpha > 0,$ we obtain
\begin{eqnarray*}
\lim_{t \to 0^+} \|\varphi(t)\| \left  \|\sum_{i \in I_1 \cup \{0\}}  \lambda_i(t) \nabla f_i(\varphi(t)) + 
\sum_{k \not \in J}  \nu_k(t) \nabla h_k(\varphi(t)) \right\| &=& 0.
\end{eqnarray*}
Since the curve $\varphi(t)$ lies in the constraint set $S,$ it follows from (c1), (c2) and the above equation that  the restriction of $f_0$ on $S$ does not satisfy the so-called {\em (weak) Palais--Smale condition}\footnote{Given a differentiable function $f \colon \mathbb{R}^n \rightarrow \mathbb{R}$ and a value $y \in \mathbb{R},$ we say that $f$  satisfies the {\em weak Palais--Smale condition} at  $y$ if any sequence $\{x^\ell\} \subset \mathbb{R}^n$ such that $f(x^\ell) \rightarrow y$ and $\|x^\ell\| \|\nabla f(x^\ell)\| \rightarrow 0$  as $\ell \to \infty$ contains a convergent subsequence (whose limit is then a critical point with critical value $y$).} at the optimal value $f^*.$ We refer the reader to the survey of Mawhin and Willem \cite{Mawhin2010} for more details about the history and genesis of the Palais--Smaile condition.
}\end{remark}

\begin{remark}{\rm
It is worth mentioning that, very recently, relying on results from real algebraic geometry, Lasserre \cite{Lasserre2015} derived {\em global optimality conditions} for polynomial optimization which generalize the local optimality conditions due to Fritz-John and Karush--Kuhn--Tucker for nonlinear optimization. 
}\end{remark}

We now study the existence of optimal solutions to the optimization problem~\eqref{PT0}. Very recently, it was proved in \cite{Dinh2014-2} that 
\eqref{PT0} has an optimal solution provided the following conditions hold:
\begin{enumerate}
\item [(i)]   all the polynomial functions $f_0, f_1, \ldots, f_p$ are convenient, and 
\item [(ii)]  the polynomial map $(f_0, f_{1}, \ldots, f_{p}) \colon {\Bbb R}^n \rightarrow {\Bbb R}^{p + 1}$ satisfies the so-called condition of non-degeneracy at infinity. (See also Definition~\ref{Definition1}.)
\end{enumerate}
We would like to mention that Theorem~\ref{Fritz-JohnTypeTheorem} (and hence Theorem~\ref{Frank-WolfeTypeTheorem} below) still holds if the (MF)$_{\infty}$ property is replaced by the condition of non-degeneracy at infinity; we leave it the reader to verify these facts. Moreover, as a first application of Theorem~\ref{Fritz-JohnTypeTheorem}, we obtain the following  result which improves Theorem~1.1 in \cite{Dinh2014-2}.

\begin{theorem}[{A Frank--Wolfe type theorem}] \label{Frank-WolfeTypeTheorem}
Let the $\mathrm{(MF)}_\infty$ property hold for the problem $\inf_{x \in S} f_0(x).$ In addition, if the polynomial $f_0$ is convenient, then the problem has at least an optimal solution.
\end{theorem}
\begin{proof} 
Suppose, the assertion of the theorem is false. Keeping the notations as in the proof of Theorem~\ref{Fritz-JohnTypeTheorem}. 
We have $\min_{j \in J} q_j < 0$ and $d_0 := d(q, \Gamma(f_0)) \ge 0.$ Let $j^* \in J$ be such that $q_{j^*} = \min_{j \in J} q_j.$ 
Since $f_0$ is convenient, there exists some $\alpha > 0$ such that $(0, \ldots, 0, \stackrel{\ j^*}{\hat{\alpha}}, 0, \ldots, 0) \in \Gamma(f_0).$
Therefore, 
\begin{eqnarray*}
0 \ \le \ d(q, \Gamma(f_0))  &=& \min \left \{\sum_{j = 1}^n q_j \kappa_j \ : \ \kappa \in \Gamma(f_0) \right \} \ \le \ q_{j^*} \alpha \ < \ 0,
\end{eqnarray*}
which is impossible. The theorem is proved.
\end{proof}

\begin{example}{\rm
Let $f_0(x, y, z) := x^2 + y^2 + z$ and $S := \{(x, y, z)  \in \mathbb{R}^3 \ : \ f_1(x, y, z)  := - z \le 0\}.$ It is easy to check that the problem $\inf_{x \in S} f_0(x)$ is bounded below and has the $\mathrm{(MF)}_\infty$ property. Since $f_0$ is convenient, 
it follows from Theorem~\ref{Frank-WolfeTypeTheorem} that the polynomial $f_0$ attains its infimum on $S;$ namely, we can see that $(0, 0, 0) \in S$ and $f^* = f(0, 0, 0).$ Notice that 
\cite[Theorem~1.1]{Dinh2014-2} cannot be applied for this example because $f_1$ is not convenient.
}\end{example}

\section{The unconstrained case} \label{UnconstrainedCase}

In the rest of this paper we assume that $S = \mathbb{R}^n.$ Then \eqref{PT0} is a unconstrained optimization problem:
$$f^* := \inf_{x \in \mathbb{R}^n} f_0(x).$$

\begin{theorem}[Fermat's theorem]  \label{FermatTypeTheorem}
Assume that the polynomial $f_0$ is bounded from below and non-degenerate at infinity. If $f_0$ does not attain its infimum $f^*,$ then there exist a point $x^* \in (\mathbb{R} \setminus \{0\})^n$ and a bad face $\Delta_0$ of $\Gamma(f_0)$ such that
$$f^* = f_{0, \Delta_0}(x^*) \quad \textrm{ and } \quad \nabla f_{0, \Delta_0}(x^*)   = 0.$$
\end{theorem}
\begin{proof}
Since $f_0$ is non-degenerate at infinity, it follows from Lemma~\ref{Lemma21} that the problem $\inf_{x \in \mathbb{R}^n} f_0(x)$ has the $\mathrm{(MF)}_\infty$ property.
Keeping the notations as in the proof of Theorem~\ref{Fritz-JohnTypeTheorem}. 
There exist a nonempty set $J \subset \{1, \ldots, n\},$ a vector $q \in \mathbb{R}^n$ with $\min_{j \in J} q_j < 0,$ and a point $x^* \in \mathbb{R}^n$ such that the conditions~(c1)-(c9) hold. 

Observe that, if $d_0 := d(q, \Gamma(f_0)) > 0$ then $f^* = 0$ and $0 \not \in \Gamma(f_0);$ hence $f(0) = 0 = f^*$ which contradicts our assumption.
By Theorem~\ref{Fritz-JohnTypeTheorem}, therefore
\begin{enumerate}
\item [(a)]   $x^*_j = 0$ if, and only if, $j \not \in J;$
\item [(b)]   $\Delta_0 := \Delta(q, \Gamma(f_0)) \subset \mathbb{R}^J$ and $d_0 = 0;$
\item [(c)]   $f^* = f_{0, \Delta_0}(x^*)$ and $\nabla f_{0, \Delta_0}(x^*) = 0.$
\end{enumerate}
(The assumption $S = \mathbb{R}^n$ yields that $I_1 = \emptyset$ and hence, by \eqref{PT4}, that $\lambda^*_0 = 1).$ Furthermore, Conditions~(c6) reads
\begin{eqnarray}\label{PT11}
\nabla f_0(\varphi(t)) +  \sum_{k \not \in J}  \nu_k(t) \nabla h_k(\varphi(t)) + \mu(t) \varphi(t) \equiv 0.
\end{eqnarray}
On the other hand, since $\Delta_0 \subset \mathbb{R}^J,$ the polynomial $f_{0, \Delta_0}$ does not depend on $x_j$ for all $j \not \in J.$ Hence, by re-assigning $x^*_j := 1$ for all $j \not \in J,$ we obtain the new point $x^* \in (\mathbb{R} \setminus \{0\})^n$ for which the property~(c) still holds. 

We next prove that $\max_{j \in J} q_j > 0.$ If it is not the case, then we have
\begin{eqnarray*}
\sum_{j = 1}^n q_j \kappa_j  \ = \ \sum_{j \in J} q_j \kappa_j &\le& 0 \quad \textrm{ for all } \quad \kappa:= (\kappa_1, \ldots, \kappa_n)  \in \Gamma(f_0) \cap \mathbb{R}^J.
\end{eqnarray*}
Since $d_0 = 0$ is the smallest value of the linear function $\sum_{j = 1}^n q_j \kappa_j$ on $\Gamma(f_0) \cap \mathbb{R}^J$ (see the equation~\eqref{PT5}), this follows that 
\begin{eqnarray*}
\kappa_{j^*} & = & 0 \quad \textrm{ for all } \quad \kappa \in \Gamma(f_0) \cap \mathbb{R}^J,
\end{eqnarray*}
where $j^* \in J$ is such that $q_{j^*} = \min_{j \in J} q_j < 0.$ Consequently, the restriction of $f_0$ on $\mathbb{R}^J$ does not depend on the variable $x_{j^*},$ and so
$\frac{\partial f_0}{\partial x_{j^*}}(\varphi(t)) \equiv 0,$ in contradiction to \eqref{PT11} because we know that $\frac{\partial h_k}{\partial x_{j^*}}(\varphi(t)) \equiv 0$ for all $k \not \in J$ 
 and $\varphi_{j^*}(t) \not \equiv 0.$

In summary, we have shown that $d(q, \Gamma(f_0)) = 0$ and $\min_{j \in J} q_j < 0 < \max_{j \in J} q_j.$ Let
$$H := \{x \in \mathbb{R}^n \ : \ \sum_{j = 1}^n q_j \kappa_j  = 0 \}.$$
Then the equality $H \cap \Gamma(f_0) = \Delta_0$ follows immediately from definitions. On the other hand, if the affine subvariety of dimension = $\dim (\Delta_0)$ spanned by $\Delta_0$ does not contain the origin, then $\Delta_0 \in \Gamma_\infty(f_0),$ 
which, together with Condition~(c) above, leads to a contradiction with the nondegenracy condition of $f_0.$ 
Hence the conditions (i) and (ii) in the definition of a bad face are fulfilled. The theorem is proved.
\end{proof} 

Let $K_0(f_0)$ be the set of critical values of $f_0$, i.e., 
\begin{eqnarray*}
K_0(f_0) &:=& \{f_0(x) \ : \ x \in \mathbb{R}^n \textrm{ and } \nabla f_{0}(x) = 0\};
\end{eqnarray*}
we also put 
\begin{eqnarray*}
\Sigma_\infty(f_0) &:=& \bigcup_{\Delta} \{f_{0, \Delta}(x) \ : \ x \in \mathbb{R}^n \textrm{ and } \nabla f_{0, \Delta}(x) = 0\},\\
\Sigma_\infty'(f_0) &:=& \bigcup_{\Delta} \{f_{0, \Delta}(x) \ : \ x \in (\mathbb{R} \setminus \{0\})^n \textrm{ and } \nabla f_{0, \Delta}(x) = 0\},
\end{eqnarray*}
where the unions are taken over all bad faces $\Delta$ of $\Gamma(f_0).$ Clearly, we have $\Sigma_\infty'(f_0) \subset \Sigma_\infty(f_0).$ Furthermore, by a semi-algebraic version of Sard's theorem (see, for example, \cite{HaHV2017}), the sets $K_0(f_0), \Sigma_\infty(f_0)$ and
$\Sigma_\infty'(f_0)$ are finite.

\begin{theorem} \label{Theorem42}
Assume that the polynomial $f_0$ is bounded from below on $\mathbb{R}^n$ and non-degenerate at infinity. We have
\begin{eqnarray*}
f^* 
&=& \min \{c \ : \ c \in K_0(f_0) \cup \Sigma_\infty(f_0) \} \\
&=& \min \{c \ : \ c \in K_0(f_0) \cup \Sigma_\infty'(f_0) \}.
\end{eqnarray*}
\end{theorem}

\begin{proof} 
We first show that 
\begin{eqnarray}\label{PT12}
f^* &\le& \min \{c \ : \ c \in K_0(f_0) \cup \Sigma_\infty(f_0) \}.
\end{eqnarray}
Indeed, it is clear that 
\begin{eqnarray*}
f^* &\le& \min \{c \ : \ c \in K_0(f_0)\}.
\end{eqnarray*}
Let $\Delta$ be a bad face of $\Gamma(f_0)$ and let $x^* \in \mathbb{R}^n$ be a critical point of $f_{0, \Delta}$ such that
\begin{eqnarray*}
f_{0, \Delta}(x^*) &=& \min \{c \ : \ c \in \Sigma_\infty(f_0) \}.
\end{eqnarray*}
By definition, there exist a nonzero vector $q \in \mathbb{R}^n$ with $\min_{j = 1, \ldots, n} q_j < 0$ such that $\Delta = \Delta(q, \Gamma(f_0))$ and 
$d(q, \Gamma(f_0)) = 0.$ Let $J$ be the smallest subset of $\{1, \ldots, n\}$ such that $\Delta \subset \mathbb{R}^J.$ 
Then  $J \ne \emptyset$ because $\Delta$ is a bad face of $\Gamma(f_0).$ Moreover, the restriction of $f_0$ on $\mathbb{R}^J$ is not constant and the polynomial $f_{0, \Delta}$ does not depend on the variables $x_j$ for $j \not \in J.$ 

Let $J' := \{j \in J \ : \ x^*_j \ne 0\}.$ If $J' = \emptyset$ or $f_0|_{\mathbb{R}^{J'}}$ is constant, then it follows from definitions that
\begin{eqnarray*}
f^* & \le & f_0(0) \ = \ f_{0, \Delta}(0) \ = \ f_{0, \Delta}(x^*)
\end{eqnarray*}
and \eqref{PT12} holds. So assume that $J' \ne \emptyset$ and $f_0|_{\mathbb{R}^{J'}}$ is not constant. Let $\Delta' := \Delta \cap \mathbb{R}^{J'}.$ Then $\Delta'$ is a closed face of 
$\Gamma(f_0|_{\mathbb{R}^{J'}}).$ Furthermore, by definition, we have
\begin{eqnarray*}
f_{0, \Delta'}(x^*) \ = \ f_{0, \Delta}(x^*).
\end{eqnarray*}
Define the monomial curve $\varphi \colon (0, 1) \rightarrow \mathbb{R}^n, t \mapsto (\varphi_1(t), \ldots, \varphi_n(t)),$ by
$$\varphi_j(t) :=
\begin{cases}
x^*_j t^{q_j} & \textrm{ if } j \in J', \\
0 & \textrm{ otherwise.}
\end{cases}$$
A simple calculation shows that
\begin{eqnarray*}
f_0(\varphi(t))  &=& f_{0, \Delta'}(x^*) + \textrm{ higher order terms in } t.
\end{eqnarray*}
Consequently, we obtain
\begin{eqnarray*}
f^* & \le & \lim_{t \to 0^+} f_0(\varphi(t)) \ = \ f_{0, \Delta'}(x^*) \ = \ f_{0, \Delta}(x^*)
\end{eqnarray*}
and \eqref{PT12} is proved.

We now assume that $f_0$ attains its infimum $f^*$ on $\mathbb{R}^n.$ Then 
\begin{eqnarray*}
f^* 
&=& \min \{c \ : \ c \in K_0(f_0)\} \\
&\ge& \min \{c \ : \ c \in K_0(f_0) \cup \Sigma_\infty'(f_0) \} \\
&\ge& \min \{c \ : \ c \in K_0(f_0) \cup \Sigma_\infty(f_0) \}.
\end{eqnarray*}
These, together with \eqref{PT12}, yield the desired relations.

Finally, we assume that $f_0$ does not attain its infimum $f^*$ on $\mathbb{R}^n.$ By Theorem~\ref{FermatTypeTheorem}, we have
\begin{eqnarray*}
f^* 
&\ge& \min \{c \ : \ c \in \Sigma_\infty'(f_0) \} \\
&\ge& \min \{c \ : \ c \in K_0(f_0) \cup \Sigma_\infty'(f_0) \} \\
&\ge& \min \{c \ : \ c \in K_0(f_0) \cup \Sigma_\infty(f_0) \}.
\end{eqnarray*}
Combining this with \eqref{PT12} again we get the desired relations. 
\end{proof}

\begin{example}{\rm
Let $f_0 \colon \mathbb{R}^2 \rightarrow \mathbb{R}$ be the polynomial defined by $f_0(x, y) := (xy - 1)^2 + x^2.$
A simple calculation shows that $(0, 0)$ is the unique critical point of $f_0,$ and so
$$K_0(f_0) = \{f_0(0, 0)\} = \{1\}.$$ 

On the other hand, by definition, the Newton polyhedron $\Gamma(f_0)$ of $f_0$ is the  triangle with the vertices at $O(0, 0), A(2, 2)$ and $B(2, 0).$ It follows that the edge $\Delta := OA$ is the unique bad face of $\Gamma(f_0)$ and that $f_{0, \Delta} = (xy - 1)^2.$  Hence, by definition again, we obtain
\begin{eqnarray*}
\Sigma_\infty(f_0) &:=& \{0, 1\} \quad \textrm{ and } \quad \Sigma_\infty'(f_0) \ := \ \{0\}.
\end{eqnarray*}
Note that, the polynomial $f_0$ is bounded from below and non-degenerate at infinity. Therefore, due to Theorem~\ref{Theorem42}, 
\begin{eqnarray*}
f^* & = &  \min \{c \ : \ c \in K_0(f_0) \cup \Sigma_\infty(f_0) \} \ = \ 0.
\end{eqnarray*}
}\end{example}

\begin{remark}{\rm
(i) In the paper \cite{Nie2006},  Nie, Demmel, and Sturmfels established sum of squares representations of positive polynomials modulo gradient ideals, i.e., the ideals generated by all the partial derivatives. Based on these representations, the authors constructed a sequence of SDP relaxations whose optimal values converge monotonically, increasing to the smallest critical value of a polynomial. Combining this fact with Theorem~\ref{Theorem42}, we can find an appropriate sequence of computationally feasible SDP relaxations, whose optimal values converge to the infimum value $f^* := \inf_{x \in \mathbb{R}^n} f_0(x).$ These facts open up the possibility of solving previously intractable polynomial optimization problems.

(ii) We do not know whether Theorem~\ref{Theorem42} can be extended to the case of optimization problems with constraints. This question will be studied in the future work.
}\end{remark}

\subsection*{Acknowledgments}
The author thanks to Jean Bernard Lasserre for useful discussions. The final version of this paper was completed while the author was visiting LAAS--CNRS in April 2017. He wishes to thank the institute and Jean Bernard Lasserre in particular for the hospitality and financial support from the European Research Council Advanced Grant for the TAMING project, No. 666981.

\end{document}